\documentclass[11pt,reqno]{amsart}

\usepackage{amsmath,amssymb,amsfonts,amsthm}
 \usepackage{amsaddr}

\newtheorem{theorem}{{\bf Theorem}}[section]
\newtheorem{lemma}[theorem]{{\bf Lemma}}

\newtheorem{example}[theorem]{Example}
\newtheorem{remark}[theorem]{Remark}
\newtheorem{definition}[theorem]{Definition}

\numberwithin{equation}{section}

\font\mathbb=msbm10 at 12pt
\newcommand{\Z}{\mbox{\mathbb Z}}

\newcommand{\R}{\mbox{\mathbb R}}

\newcommand{\bs}{\boldsymbol}

\makeatletter
\renewcommand*\env@matrix[1][*\c@MaxMatrixCols c]{%
  \hskip -\arraycolsep
  \let\@ifnextchar\new@ifnextchar
  \array{#1}}
\makeatother

\title[Hyers-Ulam stability for differential systems]{Hyers-Ulam stability for differential systems with $2\times 2$ constant coefficient matrix}
\author[anderson]{Douglas R. Anderson}
\address{Department of Mathematics, Concordia College, Moorhead, MN 56562 USA}
\email{andersod@cord.edu}
\author[Onitsuka]{Masakazu Onitsuka}
\address{Department of Applied Mathematics, Okayama University of Science, Okayama, 700-0005 Japan}
\email{onitsuka@xmath.ous.ac.jp}

\begin{document}

\begin{abstract}
We explore the Hyers-Ulam stability of perturbations for a homogeneous linear differential system with $2\times 2$ constant coefficient matrix. New necessary and sufficient conditions for the linear system to be Hyers-Ulam stable are proven, and for the first time, the best (minimal) Hyers-Ulam constant for systems is found in some cases. Several examples are provided. Obtaining the best Hyers-Ulam constant for second-order constant coefficient differential equations illustrates the applicability of the strong results.
\end{abstract}

\keywords{Perturbations, Hyers-Ulam stability, differential systems, Putzer's algorithm, minimal Hyers-Ulam constant}

\subjclass[2010]{34D10, 34A12, 34A30, 39B82}

\maketitle\thispagestyle{empty}



\section{Introduction}

Given a solution to a perturbation of a linear system, an important question, first posed by Ulam, is whether a solution to the unperturbed system exists that stays close to the perturbed solution, and if so, how close. This is frequently discussed under the rubric Hyers-Ulam stability. Of late, there is a significant interest in Hyers-Ulam stability (HUS) for differential equations and differential systems. Related to HUS for second-order differential equations, please see \cite{rpp,bp,dragicevic,fo,ym}; for third-order or fourth-order differential equations, see \cite{mp2,ugbrgvla}. In the case of higher-order differential equations, the papers \cite{mp1,mps,msp,sl} are of interest. The paper \cite{jung} established the Hyers-Ulam stability of a system of first-order linear differential equations with constant coefficients, followed by \cite{jung2,jung3}. Other recent papers on HUS and differential systems include \cite{bclo,bmppr}.
In this work, we sharpen and clarify those results, in the $2\times 2$ constant coefficient matrix case, by providing necessary and sufficient conditions for the linear system to be Hyers-Ulam stable, and we suss out specific Hyers-Ulam constants in the event of Hyers-Ulam stability, including specific cases with the minimal constant. We also provide a lower bound for the Hyers-Ulam constant, in the case of stability. The key to instability is at least one eigenvalue of the constant coefficient matrix has zero real part. We employ the direct formula for the matrix exponential via Putzer's algorithm, in place of the Jordan form. These new results significantly improve \cite[Theorem 3]{bmppr} and \cite[Corollary 2]{jung2} in the $2\times 2$ constant coefficient matrix case.

We consider the linear differential system
\begin{equation}\label{maineq}
   \bs{x}'(t)=A\bs{x}(t), \quad t\in\R,
\end{equation}
where $A$ is a $2\times 2$ constant matrix, and use the $\R^2$ maximum norm 
$$ \|\bs{v}\|_{\infty}:=\max\left\{|v_{1}|,\; |v_{2}|\right\}, \quad \bs{v}=\begin{bmatrix} v_1 \\ v_2 \end{bmatrix}, $$
with induced matrix norm
$$ \|A\|_{\infty}:=\max_{i=1,2}\left\{|a_{i1}|+|a_{i2}|\right\}, \quad A=\begin{bmatrix} a_{11} & a_{12} \\ a_{21} & a_{22} \end{bmatrix}. $$


\begin{definition}[Hyers-Ulam Stability]
Given an arbitrary $\varepsilon>0$, suppose a differentiable vector function $\bs{\phi}$ satisfies 
\begin{equation}\label{phi-eq}
 \left\|\bs{\phi}'(t)-A\bs{\phi}(t)\right\|_{\infty} \le \varepsilon, \quad t\in\R. 
\end{equation}
Then, \eqref{maineq} is Hyers-Ulam stable on $\R$ if and only if there exists a solution $\bs{x}$ of \eqref{maineq} such that
$$ \|\bs{\phi}(t)-\bs{x}(t)\|_{\infty} \le K\varepsilon, \quad t\in\R, $$
where $K$ is the Hyers-Ulam stability constant. 
\end{definition}

In the next section, the solution formula is presented by using the Putzer algorithm. Furthermore, based on that formula the lower bound for HUS constant $K$ is given. Section 3 is the core of this study. Hyers-Ulam stability is considered and some specific HUS constants are given. In Section 4, instability is considered. In addition, by combining the results obtained in Sections 2, 3, and 4, a necessary and sufficient condition for HUS is presented, and in some cases the minimal (best) HUS constant is given. In Section 5, the results obtained are applied to second-order linear differential equations and examples with the best HUS constants are presented. Concluding remarks on the significance of the results are given in Section 6.


\section{Putzer's Algorithm and Lower Bound for HUS Constants}
In this section, two technical lemmas are proven that lay the foundation for the main results to follow. Note that Lemma \ref{lower-bound} below, which provides a lower bound for the Hyers-Ulam constant in the case of Hyers-Ulam stability, is a novel result, as most studies find only upper bounds.


\begin{lemma}[Matrix Exponential]\label{distinct}
If $A$ has real and distinct eigenvalues $\lambda_1$ and $\lambda_2$, then $A$ has the form
\begin{equation}\label{Aformdistinct}
 A=\begin{bmatrix} a & b \\ c & \lambda_1+\lambda_2-a \end{bmatrix}, 
\end{equation}
and $e^{tA}$ has the form
\begin{equation}\label{distinctev}
 e^{tA} = \frac{1}{\lambda_1-\lambda_2} \begin{bmatrix} (a-\lambda_2)e^{\lambda_1 t}+(\lambda_1-a)e^{\lambda_2 t} & b\left(e^{\lambda_1 t}-e^{\lambda_2 t}\right) \\
          c\left(e^{\lambda_1 t}-e^{\lambda_2 t}\right) & (\lambda_1-a)e^{\lambda_1 t}+(a-\lambda_2)e^{\lambda_2 t}\end{bmatrix},
\end{equation}
where $bc=(\lambda_1-a)(a-\lambda_2)$.

If $A$ has a repeated eigenvalue $\lambda\in\R$, then $A$ has the form
\begin{equation}\label{Aformrepeat}
 A = \begin{bmatrix} a & b \\ c & d \end{bmatrix} = \begin{bmatrix} \lambda-\eta & b \\ c & \lambda+\eta \end{bmatrix} 
\end{equation}
for some $\eta=\lambda-a$, and $e^{tA}$ has the form
\begin{equation}\label{repeatev}
e^{tA} = e^{\lambda t}\begin{bmatrix} 1-\eta t & bt \\ ct & 1+\eta t \end{bmatrix}, 
\end{equation}
where $bc+\eta^2=0$. 

If $A$ has complex-valued eigenvalues $\lambda=\alpha\pm i\beta$ with $\beta\ne 0$ and $i=\sqrt{-1}$, then 
\begin{equation}\label{Aformcomplex}
 A=\begin{bmatrix} a & b \\ c & 2\alpha-a \end{bmatrix}, 
\end{equation}
and $e^{tA}$ has the form
\begin{equation}\label{complexev}
 e^{tA} = e^{\alpha t} \begin{bmatrix} \cos(\beta t) + \frac{a-\alpha}{\beta}\sin(\beta t) & \frac{b}{\beta}\sin(\beta t)  \\
          \frac{c}{\beta}\sin(\beta t)  & \cos(\beta t) + \frac{\alpha-a}{\beta}\sin(\beta t)  \end{bmatrix},
\end{equation}
where $(\alpha-a)^2+bc+\beta^2=0$.
\end{lemma}

\begin{proof}
Represent the $2\times 2$ matrix $A$ via the real entries 
\begin{equation}\label{genAform}
 A=\begin{bmatrix} a & b \\ c & d \end{bmatrix}. 
\end{equation}
Assuming $A$ has real and distinct eigenvalues $\lambda_1$ and $\lambda_2$, then the trace of $A$ satisfies $a+d=\lambda_1+\lambda_2$ and the determinant satisfies $ad-bc=\lambda_1\lambda_2$, so $d=\lambda_1+\lambda_2-a$ and $bc=(\lambda_1-a)(a-\lambda_2)$, via \eqref{genAform}. Using Putzer's algorithm \cite[Theorem 2.35]{kp}, we have
\[ e^{tA} = e^{\lambda_1 t}I_2 + \frac{1}{\lambda_1-\lambda_2} \left(e^{\lambda_1 t}-e^{\lambda_2 t}\right)(A-\lambda_1I_2), \]
where $I_2$ is the $2\times 2$ identity matrix. Consequently, the first result follows. 

If $A$ has a repeated eigenvalue $\lambda\in\R$, then $A$ has the form \eqref{Aformrepeat} for some $\eta=\lambda-a$, and by Putzer's algorithm, $e^{tA}$ has the form \eqref{repeatev} above, where $bc+\eta^2=0$. 

Now, consider the case where $\lambda=\alpha\pm i\beta$ with $\beta\ne 0$. Then, the trace of $A$ in \eqref{genAform} satisfies $a+d=2\alpha$ and the determinant satisfies $ad-bc=\alpha^2+\beta^2$. Again, using Putzer's algorithm \cite[Theorem 2.35]{kp}, we have
\[ e^{tA} = e^{(\alpha+i\beta)t}I_2 + \frac{1}{2i\beta} \left(e^{(\alpha+i\beta) t}-e^{(\alpha-i\beta) t}\right)(A-(\alpha+i\beta)I_2), \]
where $I_2$ is the $2\times 2$ identity matrix. Consequently, the third result follows. 
This completes all cases and the proof.
\end{proof}


\begin{lemma}[Lower Bound for HUS Constant $K$]\label{lower-bound}
Suppose that \eqref{maineq} is Hyers-Ulam stable and $A$ is invertible. If the real parts of the two eigenvalues of $A$ are both nonzero, then the minimal Hyers-Ulam constant $K$ satisfies $K\ge\left\|A^{-1}\bs{e}\right\|_{\infty}$, where $\bs{e}$ is any unit vector.
\end{lemma}

\begin{proof}
Let arbitrary $\varepsilon>0$ be given, and consider \eqref{maineq}. Assume $A$ is invertible. Set
\[ \bs{\phi}(t) = \varepsilon\left(e^{tA}-I_2\right)A^{-1}\bs{e}, \]
where $\bs{e}$ is any unit vector. Recalling that $A^{-1}$ and $e^{tA}$ commute, it follows that
\[ \bs{\phi}'(t)-A\bs{\phi}(t) = \varepsilon e^{tA}\bs{e} - \varepsilon \left(e^{tA}-I_2\right)\bs{e} = \varepsilon \bs{e}, \]
and thus $\|\bs{\phi}'(t)-A\bs{\phi}(t)\|_{\infty} = \varepsilon$; in other words, $\bs{\phi}$ satisfies \eqref{phi-eq}. Let 
\[ \bs{x}(t)=e^{tA}\bs{x}_0, \quad \bs{x}_0\in \R^2. \]
Then, $\bs{x}$ is a solution of \eqref{maineq}, and 
\[ \|\bs{\phi}(t)-\bs{x}(t)\|_{\infty} = \left\|e^{tA}\left(\varepsilon A^{-1}\bs{e}-\bs{x}_0\right)-\varepsilon A^{-1}\bs{e}\right\|_{\infty}. \]
If the condition
\[ \bs{x}_0=\varepsilon A^{-1}\bs{e} \]
is satisfied, then
\[ \|\bs{\phi}(t)-\bs{x}(t)\|_{\infty} = \varepsilon\left\|A^{-1}\bs{e}\right\|_{\infty} < \infty. \]
Conversely, we will show that $\|\bs{\phi}(t)-\bs{x}(t)\|_{\infty} < \infty$ implies $\bs{x}_0=\varepsilon A^{-1}\bs{e}$. To prove this, using reductio ad absurdum, we assume that $\|\bs{\phi}(t)-\bs{x}(t)\|_{\infty} < \infty$ on $\R$ and $\bs{y}_0\ne \bs{0}$, where
\[ \bs{y}_0=\begin{bmatrix} y_1 \\ y_2 \end{bmatrix} =\varepsilon A^{-1}\bs{e}-\bs{x}_0, \quad \bs{e} = \begin{bmatrix} u_1 \\ u_2 \end{bmatrix}. \]

(I) Suppose $A$ has real and distinct nonzero eigenvalues $\lambda_1$ and $\lambda_2$ with $\lambda_1 > \lambda_2$. By Lemma \ref{distinct}, $A$ has the form \eqref{Aformdistinct}, and $e^{tA}$ has the form \eqref{distinctev},
where $bc=(\lambda_1-a)(a-\lambda_2)$. 

(i) Suppose $\lambda_1 >0$, $\lambda_1 \ne a \ne \lambda_2$, with $b,c>0$ or $b,c<0$. Then
\begin{align*}
 &\lim_{t\rightarrow\infty}e^{-\lambda_1 t}\|\bs{\phi}(t)-\bs{x}(t)\|_{\infty} = \lim_{t\rightarrow\infty}e^{-\lambda_1 t}\left\|e^{tA}\bs{y}_0-\varepsilon A^{-1}\bs{e}\right\|_{\infty}\\
 &= \lim_{t\rightarrow\infty}\Bigg\|\frac{1}{\lambda_1-\lambda_2} \begin{bmatrix} a-\lambda_2+(\lambda_1-a)e^{(\lambda_2-\lambda_1)t} & b\left(1-e^{(\lambda_2-\lambda_1)t}\right) \\
          c\left(1-e^{(\lambda_2-\lambda_1)t}\right) & \lambda_1-a+(a-\lambda_2)e^{(\lambda_2-\lambda_1)t}\end{bmatrix}\bs{y}_0\\
 &\quad -\varepsilon e^{-\lambda_1 t}A^{-1}\bs{e}\Bigg\|_{\infty} \\
 &= \frac{1}{\lambda_1-\lambda_2} \Bigg\|\begin{bmatrix} a-\lambda_2 & b \\
          c & \lambda_1-a \end{bmatrix}\bs{y}_0 \Bigg\|_{\infty} > 0.
\end{align*}
This implies that $\lim_{t\rightarrow\infty}\|\bs{\phi}(t)-\bs{x}(t)\|_{\infty} = \infty$. This is a contradiction. 

(ii) Suppose $a=\lambda_1> 0$ and $b=0$ or $c=0$. First, suppose $a=\lambda_1>0$ and $b\ne 0=c$. Then
\begin{align*}
 &\lim_{t\rightarrow\infty}e^{-\lambda_1 t}\|\bs{\phi}(t)-\bs{x}(t)\|_{\infty}\\
 &= \lim_{t\rightarrow\infty}\Bigg\|\frac{1}{\lambda_1-\lambda_2} \begin{bmatrix} a-\lambda_2 & b\left(1-e^{(\lambda_2-\lambda_1)t}\right) \\
          0 & (a-\lambda_2)e^{(\lambda_2-\lambda_1)t}\end{bmatrix}\bs{y}_0 -\varepsilon e^{-\lambda_1 t}A^{-1}\bs{e}\Bigg\|_{\infty} \\
 &= \frac{1}{\lambda_1-\lambda_2} \Bigg\|\begin{bmatrix} a-\lambda_2 & b \\
          0 & 0 \end{bmatrix}\bs{y}_0 \Bigg\|_{\infty} > 0.
\end{align*}
This implies that $\lim_{t\rightarrow\infty}\|\bs{\phi}(t)-\bs{x}(t)\|_{\infty} = \infty$. 

Next, suppose $a=\lambda_1>0$ and $b=0\ne c$. If $y_1\ne0$, then
\begin{align*}
 &\lim_{t\rightarrow\infty}e^{-\lambda_1 t}\|\bs{\phi}(t)-\bs{x}(t)\|_{\infty}\\
 &= \lim_{t\rightarrow\infty}\Bigg\|\frac{1}{\lambda_1-\lambda_2} \begin{bmatrix} a-\lambda_2 & 0 \\
          c\left(1-e^{(\lambda_2-\lambda_1)t}\right) & (a-\lambda_2)e^{(\lambda_2-\lambda_1)t}\end{bmatrix}\bs{y}_0 -\varepsilon e^{-\lambda_1 t}A^{-1}\bs{e}\Bigg\|_{\infty} \\
 &= \frac{1}{\lambda_1-\lambda_2} \Bigg\|\begin{bmatrix} a-\lambda_2 & 0 \\
          c & 0 \end{bmatrix}\begin{bmatrix} y_1 \\ y_2 \end{bmatrix} \Bigg\|_{\infty} = \frac{1}{\lambda_1-\lambda_2}\max\{|(a-\lambda_2)y_1|, |cy_1| \} > 0,
\end{align*}
and thus, $\lim_{t\rightarrow\infty}\|\bs{\phi}(t)-\bs{x}(t)\|_{\infty} = \infty$. If $y_1=0$, then $y_2\ne0$ and
\begin{align*}
 &\|\bs{\phi}(t)-\bs{x}(t)\|_{\infty}\\
 &= \Bigg\|\frac{1}{\lambda_1-\lambda_2} \begin{bmatrix} (a-\lambda_2)e^{\lambda_1 t} & 0 \\
          c\left(e^{\lambda_1 t}-e^{\lambda_2 t}\right) & (a-\lambda_2)e^{\lambda_2 t}\end{bmatrix}\begin{bmatrix} 0 \\ y_2 \end{bmatrix} - \frac{\varepsilon}{\lambda_1\lambda_2}\begin{bmatrix} \lambda_2u_1 \\ -cu_1+\lambda_1u_2 \end{bmatrix}\Bigg\|_{\infty}\\
 &= \Bigg\|\begin{bmatrix} - \frac{\varepsilon u_1}{\lambda_1} \\ e^{\lambda_2 t} y_2- \frac{\varepsilon(-cu_1+\lambda_1u_2)}{\lambda_1\lambda_2} \end{bmatrix}\Bigg\|_{\infty} 
= \max\left\{\frac{\varepsilon |u_1|}{\lambda_1}, \left|e^{\lambda_2 t} y_2- \frac{\varepsilon(-cu_1+\lambda_1u_2)}{\lambda_1\lambda_2}\right| \right\}.
\end{align*}
This implies that $\lim_{t\rightarrow\infty}\|\bs{\phi}(t)-\bs{x}(t)\|_{\infty} = \infty$ if $\lambda_2>0$, and $\lim_{t\rightarrow-\infty}\|\bs{\phi}(t)-\bs{x}(t)\|_{\infty} = \infty$ if $\lambda_2<0$. 

Next, suppose $a=\lambda_1>0$ and $b=c=0$. If $y_1\ne0$, then
\[
 \lim_{t\rightarrow\infty}e^{-\lambda_1 t}\|\bs{\phi}(t)-\bs{x}(t)\|_{\infty} = \frac{1}{\lambda_1-\lambda_2} \Bigg\|\begin{bmatrix} a-\lambda_2 & 0 \\
          0 & 0 \end{bmatrix}\begin{bmatrix} y_1 \\ y_2 \end{bmatrix} \Bigg\|_{\infty} = |y_1| > 0.
\]
This implies $\lim_{t\rightarrow\infty}\|\bs{\phi}(t)-\bs{x}(t)\|_{\infty} = \infty$. If $y_1=0$, then $y_2\ne0$ and
\begin{align*}
 \|\bs{\phi}(t)-\bs{x}(t)\|_{\infty} &= \Bigg\|\begin{bmatrix} e^{\lambda_1 t} & 0 \\
          0 & e^{\lambda_2 t} \end{bmatrix}\begin{bmatrix} 0 \\ y_2 \end{bmatrix} - \frac{\varepsilon}{\lambda_1\lambda_2}\begin{bmatrix} \lambda_2u_1 \\ \lambda_1u_2 \end{bmatrix}\Bigg\|_{\infty}\\
 &= \Bigg\|\begin{bmatrix} - \frac{\varepsilon u_1}{\lambda_1} \\ e^{\lambda_2 t} y_2- \frac{\varepsilon u_2}{\lambda_2} \end{bmatrix}\Bigg\|_{\infty}
  = \max\left\{\frac{\varepsilon |u_1|}{\lambda_1}, \left|e^{\lambda_2 t} y_2- \frac{\varepsilon u_2}{\lambda_2}\right| \right\}.
\end{align*}
This implies that $\lim_{t\rightarrow\infty}\|\bs{\phi}(t)-\bs{x}(t)\|_{\infty} = \infty$ if $\lambda_2>0$, and $\lim_{t\rightarrow-\infty}\|\bs{\phi}(t)-\bs{x}(t)\|_{\infty} = \infty$ if $\lambda_2<0$. Therefore, there is a contradiction in all cases. Hence, in the case $a=\lambda_1> 0$ and $b=0$ or $c=0$, we can conclude that $\bs{y}_0 = \bs{0}$.

(iii) Suppose $\lambda_1> 0$, $a=\lambda_2$ and $b=0$ or $c=0$. First, suppose $\lambda_1> 0$, $a=\lambda_2$ and $b=0\ne c$. Then
\[
\lim_{t\rightarrow\infty}e^{-\lambda_1 t}\|\bs{\phi}(t)-\bs{x}(t)\|_{\infty} = \frac{1}{\lambda_1-\lambda_2} \Bigg\|\begin{bmatrix} 0 & 0 \\
          c & \lambda_1-a \end{bmatrix}\bs{y}_0 \Bigg\|_{\infty} > 0,
\]
and thus, $\lim_{t\rightarrow\infty}\|\bs{\phi}(t)-\bs{x}(t)\|_{\infty} = \infty$.

Next, suppose $\lambda_1> 0$, $a=\lambda_2$ and $b\ne 0=c$. If $y_2\ne0$, then
\[ \lim_{t\rightarrow\infty}e^{-\lambda_1 t}\|\bs{\phi}(t)-\bs{x}(t)\|_{\infty} = \frac{1}{\lambda_1-\lambda_2} \Bigg\|\begin{bmatrix} 0 & b \\
          0 & \lambda_1-a \end{bmatrix}\bs{y}_0 \Bigg\|_{\infty} > 0.\]
This implies $\lim_{t\rightarrow\infty}\|\bs{\phi}(t)-\bs{x}(t)\|_{\infty} = \infty$. If $y_2=0$, then $y_1\ne0$ and
\[
\|\bs{\phi}(t)-\bs{x}(t)\|_{\infty} = \max\left\{\left|e^{\lambda_2 t} y_1- \frac{\varepsilon(\lambda_1u_1-bu_2)}{\lambda_1\lambda_2}\right|, \frac{\varepsilon |u_2|}{\lambda_1} \right\}.
\]
This implies that $\lim_{t\rightarrow\infty}\|\bs{\phi}(t)-\bs{x}(t)\|_{\infty} = \infty$ if $\lambda_2>0$, and $\lim_{t\rightarrow-\infty}\|\bs{\phi}(t)-\bs{x}(t)\|_{\infty} = \infty$ if $\lambda_2<0$. 

Next, suppose $\lambda_1> 0$, $a=\lambda_2>0$ and $b=c=0$. If $y_2\ne0$, then
\[
 \lim_{t\rightarrow\infty}e^{-\lambda_1 t}\|\bs{\phi}(t)-\bs{x}(t)\|_{\infty} = \frac{1}{\lambda_1-\lambda_2} \Bigg\|\begin{bmatrix} 0 & 0 \\
          0 & \lambda_1-a \end{bmatrix}\begin{bmatrix} y_1 \\ y_2 \end{bmatrix} \Bigg\|_{\infty} = |y_2| > 0.
\]
This implies that $\lim_{t\rightarrow\infty}\|\bs{\phi}(t)-\bs{x}(t)\|_{\infty} = \infty$. If $y_2=0$, then $y_1\ne0$ and
\[
 \|\bs{\phi}(t)-\bs{x}(t)\|_{\infty} = \max\left\{\left|e^{\lambda_2 t} y_1- \frac{\varepsilon u_1}{\lambda_2}\right|, \frac{\varepsilon |u_2|}{\lambda_1} \right\}.
\]
This implies that $\lim_{t\rightarrow\infty}\|\bs{\phi}(t)-\bs{x}(t)\|_{\infty} = \infty$ if $\lambda_2>0$, and $\lim_{t\rightarrow-\infty}\|\bs{\phi}(t)-\bs{x}(t)\|_{\infty} = \infty$ if $\lambda_2<0$. Hence, in the case $\lambda_1> 0$, $a=\lambda_2$ and $b=0$ or $c=0$, we can conclude that $\bs{y}_0 = \bs{0}$.

Using the same method as above, we can see that $\bs{y}_0 = \bs{0}$ also holds for the cases $\lambda_2 <0$, $\lambda_1 \ne a \ne \lambda_2$, with $b,c>0$ or $b,c<0$; $a=\lambda_2< 0$ and $b=0$ or $c=0$; $\lambda_2< 0$, $a=\lambda_1$ and $b=0$ or $c=0$. Hence, if $A$ has real and distinct nonzero eigenvalues $\lambda_1$ and $\lambda_2$ with $\lambda_1 > \lambda_2$, then $\bs{y}_0 = \bs{0}$. 

(II) Suppose $A$ has a repeated nonzero eigenvalue $\lambda$. By Lemma \ref{distinct}, $A$ has the form \eqref{Aformrepeat},
and $e^{tA}$ has the form \eqref{repeatev}, where $\lambda\ne 0$ and $bc+\eta^2=0$, for some $\eta=\lambda-a$. 

(iv) Suppose $\lambda \ne a$ with $b,c>0$ or $b,c<0$. Then
\[
\lim_{\lambda t\rightarrow\infty}\frac{e^{-\lambda t}}{t}\|\bs{\phi}(t)-\bs{x}(t)\|_{\infty} = \Bigg\|\begin{bmatrix} -\eta & b \\
          c & \eta \end{bmatrix}\bs{y}_0 \Bigg\|_{\infty} > 0.
\]
This implies that $\lim_{\lambda t\rightarrow\infty}\|\bs{\phi}(t)-\bs{x}(t)\|_{\infty} = \infty$, and thus, $\bs{y}_0 = \bs{0}$.

(v) Suppose $\lambda = a$ with $b=0$ or $c=0$. Then
\[
\|\bs{\phi}(t)-\bs{x}(t)\|_{\infty} = \Bigg\|e^{\lambda t}\begin{bmatrix} y_1+bty_2 \\ cty_1+y_2 \end{bmatrix} - \frac{\varepsilon}{\lambda^2}\begin{bmatrix} \lambda u_1-bu_2 \\ -cu_1+\lambda u_2 \end{bmatrix}\Bigg\|_{\infty} > 0.
\]
From $\bs{y}_0\ne \bs{0}$, $\lim_{\lambda t\rightarrow\infty}\|\bs{\phi}(t)-\bs{x}(t)\|_{\infty} = \infty$, and thus, $\bs{y}_0 = \bs{0}$. Hence, if $A$ has a repeated nonzero eigenvalue $\lambda$, then $\bs{y}_0 = \bs{0}$. 

(III) Suppose $A$ has complex-valued eigenvalues $\lambda=\alpha\pm i\beta$ with $\alpha\ne0\ne\beta$. By Lemma \ref{distinct}, $A$ has the form \eqref{Aformcomplex},
and $e^{tA}$ has the form \eqref{complexev},
where $(\alpha-a)^2+bc+\beta^2=0$. Let
\[ t_n=\frac{2n\pi}{|\beta|}, \quad n\in \Z. \]
Then
\[
 e^{t_nA} = e^{\alpha t_n} I_2,
\]
and thus,
\[ \lim_{\alpha n\rightarrow\infty}e^{-\alpha t_n}\|\bs{\phi}(t_n)-\bs{x}(t_n)\|_{\infty} = \lim_{\alpha n\rightarrow\infty}\left\|\bs{y}_0-\varepsilon e^{-\alpha t_n}A^{-1}\bs{e}\right\|_{\infty} = \|\bs{y}_0\|_{\infty} \ne 0 \]
This implies that $\lim_{\alpha n\rightarrow\infty}\|\bs{\phi}(t_n)-\bs{x}(t_n)\|_{\infty} = \infty$. This is a contradiction. Hence, $\bs{y}_0 = \bs{0}$. 

To summarize the above facts, when the real parts of the two eigenvalues of $A$ are both nonzero the following holds: $\|\bs{\phi}(t)-\bs{x}(t)\|_{\infty}$ is bounded on $\R$ if only if
\[ \bs{x}_0=\varepsilon A^{-1}\bs{e} \]
holds. This means that
\[ \bs{x}(t)=\varepsilon e^{tA} A^{-1}\bs{e} \]
is the unique solution of \eqref{maineq} satisfying $\|\bs{\phi}(t)-\bs{x}(t)\|_{\infty}<\infty$ on $\R$. Consequently, 
\[ \|\bs{\phi}(t)-\bs{x}(t)\|_{\infty} = \varepsilon\left\|A^{-1}\bs{e}\right\|_{\infty} \]
holds on $\R$, and this says that the minimal Hyers-Ulam constant $K$ is greater than or equal to $\left\|A^{-1}\bs{e}\right\|_{\infty}$. This ends the proof.
\end{proof}


\section{Hyers-Ulam Stability Constants}

\begin{theorem}[Real, Nonzero, Same Sign $\implies$ Stable]\label{thm-cases}
If $A$ has distinct, positive eigenvalues $\lambda_1 > \lambda_2 > 0$, then \eqref{maineq} is Hyers-Ulam stable on $\R$, with HUS constant
\[ K= \begin{cases} \left\|A^{-1}\right\|_{\infty} : \lambda_1 \ge a \ge \lambda_2>0, \\
  \max\left\{ \frac{\lambda_1+\lambda_2+|b|-a}{\lambda_1\lambda_2}, \frac{a+|c|+2(\lambda_2-a)\left(\frac{\lambda_1-a}{\lambda_2-a}\right)^{\frac{-\lambda_2}{\lambda_1-\lambda_2}}}{\lambda_1\lambda_2}\right\} : \lambda_1 > \lambda_2 > a, \\
  \max\left\{\frac{\lambda_1+\lambda_2+|b|-a+2(a-\lambda_1)\left(\frac{a-\lambda_2}{a-\lambda_1}\right)^{\frac{-\lambda_2}{\lambda_1-\lambda_2}}}{\lambda_1 \lambda_2},\; \frac{a+|c|}{\lambda_1\lambda_2}\right\} : a > \lambda_1 > \lambda_2. \end{cases} \]
If $A$ has distinct, negative eigenvalues $0> \lambda_1 > \lambda_2$, then \eqref{maineq} is Hyers-Ulam stable on $\R$, with HUS constant
\[ K= \begin{cases} \left\|A^{-1}\right\|_{\infty} : 0>\lambda_1 \ge a \ge \lambda_2, \\
  \max\left\{ \frac{|\lambda_1|+|\lambda_2|+a+|b|+2(\lambda_1-a)\left(\frac{\lambda_1-a}{\lambda_2-a}\right)^{\frac{\lambda_2}{\lambda_1-\lambda_2}}}{\lambda_1\lambda_2}, \frac{|a|+|c|}{\lambda_1\lambda_2}\right\} : \lambda_1 > \lambda_2 > a, \\
  \max\left\{\frac{|\lambda_1|+|\lambda_2|+a+|b|}{\lambda_1\lambda_2},\; \frac{-a+|c|+2(a-\lambda_2)\left(\frac{a-\lambda_1}{a-\lambda_2}\right)^{\frac{-\lambda_2}{\lambda_1-\lambda_2}}}{\lambda_1\lambda_2}\right\}: a > \lambda_1 > \lambda_2. \end{cases} \]
If $A$ has a repeated eigenvalue $\lambda\ne 0$, then $A$ has the form
$$ A = \begin{bmatrix} a & b \\ c & d \end{bmatrix} = \begin{bmatrix} \lambda-\eta & b \\ c & \lambda+\eta \end{bmatrix} $$
for some $\eta=\lambda-a$, where $bc+\eta^2=0$, and \eqref{maineq} is Hyers-Ulam stable on $\R$, with HUS constant
\[ K= \begin{cases} \left\|A^{-1}\right\|_{\infty} : \eta=0 \; (bc=0), \\
  \max\left\{\frac{\lambda+\eta+|b|}{\lambda^2},\; \frac{\lambda+|c|+\eta(-1+2e^{-\frac{\lambda}{\eta}})}{\lambda^2}\right\} : \eta>0, \\
  \max\left\{\frac{\lambda+|\eta|+|c|}{\lambda^2},\; \frac{\lambda+|b|+|\eta|(-1+2e^{-\frac{\lambda}{|\eta|}})}{\lambda^2}\right\} : \eta<0 \end{cases} \]
when $\lambda>0$, and
\[ K= \begin{cases} \left\|A^{-1}\right\|_{\infty} : \eta=0 \; (bc=0), \\
  \max\left\{\frac{|\lambda|+\eta+|c|}{\lambda^2},\; \frac{|\lambda|+|b|+\eta(-1+2e^{\frac{\lambda}{\eta}})}{\lambda^2}\right\} : \eta>0, \\
  \max\left\{\frac{|\lambda|+|\eta|+|b|}{\lambda^2},\; \frac{|\lambda|+|c|+|\eta|(-1+2e^{-\frac{\lambda}{\eta}})}{\lambda^2}\right\} : \eta<0 \end{cases} \]
when $\lambda<0$.	
\end{theorem}

\begin{proof}
If $A$ has distinct, positive eigenvalues $\lambda_1 > \lambda_2 > 0$, then $A$ has the form \eqref{Aformdistinct},
and $e^{tA}$ has the form \eqref{distinctev}, where $bc=(\lambda_1-a)(a-\lambda_2)$.
Given an arbitrary $\varepsilon>0$, suppose a vector function $\bs{\phi}$ satisfies 
\[ \left\|\bs{\phi}'(t)-A\bs{\phi}(t)\right\|_{\infty} \le \varepsilon, \quad t\in\R. \]
Then, there exists a vector function $\bs{q}$ satisfying $\|\bs{q}(s)\|_{\infty} \le \varepsilon$ for all $s\in\R$, such that
$\bs{\phi}'(t)-A\bs{\phi}(t)=\bs{q}(t)$; by the variation of parameters formula,
$$ \bs{\phi}(t)=e^{tA}\bs{\phi}_0 + \int_0^t e^{(t-s)A}\bs{q}(s)ds, $$
for $\bs{\phi}(0)=\bs{\phi}_0$. Note that
\[ \lim_{t\rightarrow\infty}e^{-tA}\bs{\phi}(t) = \bs{\phi}_0 + \int_0^{\infty} e^{-sA}\bs{q}(s)ds \]
is well defined (finite) in this case, as both eigenvalues are positive.
As the general solution of \eqref{maineq} is $\bs{x}(t)=e^{tA}\bs{x}_0$ for constant vector $\bs{x}_0$, pick 
$$ \bs{x}(t)=e^{tA}\left(\bs{\phi}_0 + \int_0^{\infty}e^{-sA}\bs{q}(s)ds\right), \qquad \bs{x}_0=\bs{\phi}_0 + \int_0^{\infty}e^{-sA}\bs{q}(s)ds. $$
Then, we have
\begin{align*}
 &\|\bs{\phi}(t)-\bs{x}(t)\|_{\infty} \\
 &= \left\|-\int_t^{\infty}e^{(t-s)A}\bs{q}(s)ds\right\|_{\infty} = \frac{1}{\lambda_1-\lambda_2}\times \\
 &\bigg\|\int_t^{\infty} \begin{bmatrix} \left\{(a-\lambda_2)e^{\lambda_1(t-s)}+(\lambda_1-a)e^{\lambda_2(t-s)}\right\}q_1(s)
+b\left\{e^{\lambda_1(t-s)}-e^{\lambda_2(t-s)}\right\}q_2(s) \\ c\left\{e^{\lambda_1(t-s)}-e^{\lambda_2(t-s)}\right\}q_1(s)
+ \left\{(\lambda_1-a)e^{\lambda_1(t-s)}+(a-\lambda_2)e^{\lambda_2(t-s)}\right\}q_2(s)
                          \end{bmatrix} ds\bigg\|_{\infty} \\
 &\le \varepsilon\max\begin{cases} \frac{1}{\lambda_1-\lambda_2}\int_t^{\infty} \left(\left|(a-\lambda_2)e^{\lambda_1(t-s)}+(\lambda_1-a)e^{\lambda_2(t-s)}\right|
+|b|\left|e^{\lambda_1(t-s)}-e^{\lambda_2(t-s)}\right|\right)ds, \\ 
                                   \frac{1}{\lambda_1-\lambda_2}\int_t^{\infty} \left(|c|\left|e^{\lambda_1(t-s)}-e^{\lambda_2(t-s)}\right|+\left|(\lambda_1-a)e^{\lambda_1(t-s)}+(a-\lambda_2)e^{\lambda_2(t-s)}\right|\right)ds. \end{cases}
\end{align*}
We now proceed with cases depending on the relative size of $a$ with respect to $\lambda_1$ and $\lambda_2$, and on the signs of $b$ and $c$, recalling that $bc=(\lambda_1-a)(a-\lambda_2)$.

(i) Suppose $\lambda_1 > a > \lambda_2>0$, with $b,c>0$ or $b,c<0$. First, assume $b,c>0$. Then,
\begin{align*}
&\frac{1}{\lambda_1-\lambda_2}\int_t^{\infty} \left(\left|(a-\lambda_2)e^{\lambda_1(t-s)}+(\lambda_1-a)e^{\lambda_2(t-s)}\right|
+|b|\left|e^{\lambda_1(t-s)}-e^{\lambda_2(t-s)}\right|\right)ds \\
&=\frac{1}{\lambda_1-\lambda_2}\int_t^{\infty} \left((a-\lambda_2)e^{\lambda_1(t-s)}+(\lambda_1-a)e^{\lambda_2(t-s)}
+b\left(e^{\lambda_2(t-s)}-e^{\lambda_1(t-s)}\right)\right)ds \\
&=\frac{\lambda_1+\lambda_2+b-a}{\lambda_1\lambda_2},
\end{align*}
and
\begin{align*}
&\frac{1}{\lambda_1-\lambda_2}\int_t^{\infty} \left(|c|\left|e^{\lambda_1(t-s)}-e^{\lambda_2(t-s)}\right|+\left|(\lambda_1-a)e^{\lambda_1(t-s)}+(a-\lambda_2)e^{\lambda_2(t-s)}\right|\right)ds \\
&=\frac{1}{\lambda_1-\lambda_2}\int_t^{\infty} \left(c\left(e^{\lambda_2(t-s)}-e^{\lambda_1(t-s)}\right)+(\lambda_1-a)e^{\lambda_1(t-s)}+(a-\lambda_2)e^{\lambda_2(t-s)}\right)ds \\
&=\frac{a+c}{\lambda_1\lambda_2}.
\end{align*}
Consequently, we have
\[ \|\bs{\phi}(t)-\bs{x}(t)\|_{\infty} \le \varepsilon\max
   \begin{cases} \frac{\lambda_1+\lambda_2+b-a}{\lambda_1\lambda_2}, \\ \frac{a+c}{\lambda_1\lambda_2} \end{cases} = \varepsilon\left\|A^{-1}\right\|_{\infty}, \]
and
\[ \|\bs{\phi}(t)-\bs{x}(t)\|_{\infty} \le \varepsilon\max
   \begin{cases} \frac{\lambda_1+\lambda_2-b-a}{\lambda_1\lambda_2}, \\ \frac{a-c}{\lambda_1\lambda_2} \end{cases} = \varepsilon\left\|A^{-1}\right\|_{\infty} \]
if $b,c<0$. Therefore, for the case $\lambda_1 > a > \lambda_2>0$ and $b,c>0$ or $b,c<0$, we have that \eqref{maineq} is Hyers-Ulam stable, with Hyers-Ulam stability constant
$$ K_{(i)} = \max\begin{cases} \frac{\lambda_1+\lambda_2+|b|-a}{\lambda_1\lambda_2}, \\ \frac{a+|c|}{\lambda_1\lambda_2} \end{cases} = \varepsilon\left\|A^{-1}\right\|_{\infty}. $$

(ii) Suppose $\lambda_1 > \lambda_2 > a$, with $b<0<c$ or $c<0<b$. First, assume $b<0<c$. Then,
\begin{align*}
&\frac{1}{\lambda_1-\lambda_2}\int_t^{\infty} \left(\left|(a-\lambda_2)e^{\lambda_1(t-s)}+(\lambda_1-a)e^{\lambda_2(t-s)}\right|
+|b|\left|e^{\lambda_1(t-s)}-e^{\lambda_2(t-s)}\right|\right)ds \\
&=\frac{1}{\lambda_1-\lambda_2}\int_t^{\infty} \left((a-\lambda_2)e^{\lambda_1(t-s)}+(\lambda_1-a)e^{\lambda_2(t-s)}
+|b|\left(e^{\lambda_2(t-s)}-e^{\lambda_1(t-s)}\right)\right)ds \\
&=\frac{\lambda_1+\lambda_2+|b|-a}{\lambda_1\lambda_2},
\end{align*}
and
\begin{align*}
&\frac{1}{\lambda_1-\lambda_2}\int_t^{\infty} \left(|c|\left|e^{\lambda_1(t-s)}-e^{\lambda_2(t-s)}\right|+\left|(\lambda_1-a)e^{\lambda_1(t-s)}+(a-\lambda_2)e^{\lambda_2(t-s)}\right|\right)ds \\
&=\frac{1}{\lambda_1-\lambda_2}\int_t^{t+\frac{\ln\left(\frac{\lambda_1-a}{\lambda_2-a}\right)}{\lambda_1-\lambda_2}} \left(c\left(e^{\lambda_2(t-s)}-e^{\lambda_1(t-s)}\right)+(\lambda_1-a)e^{\lambda_1(t-s)}+(a-\lambda_2)e^{\lambda_2(t-s)}\right)ds \\
&+\frac{1}{\lambda_1-\lambda_2}\int_{t+\frac{\ln\left(\frac{\lambda_1-a}{\lambda_2-a}\right)}{\lambda_1-\lambda_2}}^{\infty} \left(c\left(e^{\lambda_2(t-s)}-e^{\lambda_1(t-s)}\right)-(\lambda_1-a)e^{\lambda_1(t-s)}-(a-\lambda_2)e^{\lambda_2(t-s)}\right)ds \\
&=\frac{a+c+2(\lambda_2-a)\left(\frac{\lambda_1-a}{\lambda_2-a}\right)^{\frac{-\lambda_2}{\lambda_1-\lambda_2}}}{\lambda_1\lambda_2}.
\end{align*}
Consequently, we have
\[ \|\bs{\phi}(t)-\bs{x}(t)\|_{\infty} \le \varepsilon\max
   \begin{cases} \frac{\lambda_1+\lambda_2+|b|-a}{\lambda_1\lambda_2}, \\ \frac{a+c+2(\lambda_2-a)\left(\frac{\lambda_1-a}{\lambda_2-a}\right)^{\frac{-\lambda_2}{\lambda_1-\lambda_2}}}{\lambda_1\lambda_2}, \end{cases} \]
and
\[ \|\bs{\phi}(t)-\bs{x}(t)\|_{\infty} \le \varepsilon\max
   \begin{cases} \frac{\lambda_1+\lambda_2+b-a}{\lambda_1\lambda_2}, \\ \frac{a+|c|+2(\lambda_2-a)\left(\frac{\lambda_1-a}{\lambda_2-a}\right)^{\frac{-\lambda_2}{\lambda_1-\lambda_2}}}{\lambda_1\lambda_2} \end{cases} \]
if $c<0<b$. Therefore, for the case $\lambda_1 > \lambda_2 > a$ and $b<0<c$ or $c<0<b$, we have that \eqref{maineq} is Hyers-Ulam stable, with Hyers-Ulam stability constant
$$ K_{(ii)} = \max\left\{ \frac{\lambda_1+\lambda_2+|b|-a}{\lambda_1\lambda_2}, \frac{a+|c|+2(\lambda_2-a)\left(\frac{\lambda_1-a}{\lambda_2-a}\right)^{\frac{-\lambda_2}{\lambda_1-\lambda_2}}}{\lambda_1\lambda_2}\right\}. $$

(iii) Suppose $a > \lambda_1 > \lambda_2$, with $b<0<c$ or $c<0<b$. 
The calculations are similar to those in case (ii) and are omitted. Therefore, for the case $a>\lambda_1 > \lambda_2$ and $b<0<c$ or $c<0<b$, we have that \eqref{maineq} is Hyers-Ulam stable, with Hyers-Ulam stability constant
$$ K_{(iii)} = \max\left\{\frac{\lambda_1+\lambda_2+|b|-a+2(a-\lambda_1)\left(\frac{a-\lambda_2}{a-\lambda_1}\right)^{\frac{-\lambda_2}{\lambda_1-\lambda_2}}}{\lambda_1 \lambda_2},\; \frac{a+|c|}{\lambda_1\lambda_2}\right\}. $$

(iv) Suppose $a=\lambda_1$ and $b=0$ or $c=0$. 
The calculations are straightforward. Therefore, for the case $a=\lambda_1 > \lambda_2$ and $b=0$ or $c=0$, we have that \eqref{maineq} is Hyers-Ulam stable, with Hyers-Ulam stability constant
$$ K_{(iv)} = \max\left\{\frac{\lambda_1+|c|}{\lambda_1\lambda_2},\;\frac{\lambda_2+|b|}{\lambda_1\lambda_2} \right\} = \varepsilon\left\|A^{-1}\right\|_{\infty}. $$

(v) Suppose $a=\lambda_2$ and $b=0$ or $c=0$. 
We leave the details to the reader. Therefore, for the case $\lambda_1 > \lambda_2=a$ and $b=0$ or $c=0$, we have that \eqref{maineq} is Hyers-Ulam stable, with Hyers-Ulam stability constant
$$ K_{(v)} = \max\left\{\frac{\lambda_1+|b|}{\lambda_1\lambda_2},\;\frac{\lambda_2+|c|}{\lambda_1\lambda_2} \right\} = \varepsilon\left\|A^{-1}\right\|_{\infty}. $$
The proof for the case where $A$ has distinct, negative eigenvalues $0> \lambda_1 > \lambda_2$, is similar to the distinct, positive case, and is omitted, as is the case where there is a repeated, nonzero eigenvalue. To help the reader understand, we note that the exact solution should be chosen as
$$ \bs{x}(t)=e^{tA}\left(\bs{\phi}_0 - \int_{-\infty}^0 e^{-sA}\bs{q}(s)ds\right), \qquad \bs{x}_0=\bs{\phi}_0 - \int_{-\infty}^0 e^{-sA}\bs{q}(s)ds $$
if $A$ has distinct, negative eigenvalues $0> \lambda_1 > \lambda_2$, or $A$ has a repeated, negative eigenvalue $\lambda < 0$. 
This ends the proof.
\end{proof}

Using a different method from that above, we can get a result that includes the case where the two real eigenvalues are of different sign.

\begin{theorem}[Distinct Nonzero Real $\implies$ Stable]\label{thm-distinct-Nonzero}
If $A$ has distinct nonzero real eigenvalues $\lambda_1 > \lambda_2$, then $A$ has the form \eqref{Aformdistinct}, and \eqref{maineq} is Hyers-Ulam stable on $\R$, with HUS constant
$$ K = \frac{|\lambda_2|\max\left\{|a-\lambda_2|+|b|, |c|+|\lambda_1-a|\right\} + |\lambda_1|\max\left\{|\lambda_1-a|+|b|, |c|+|a-\lambda_2|\right\}}{|\lambda_1\lambda_2|(\lambda_1-\lambda_2)}. $$
\end{theorem}

\begin{proof}
If $\lambda_1 > \lambda_2$ are distinct nonzero real eigenvalues of $A$, then $A$ has form \eqref{Aformdistinct}; by \eqref{distinctev}, 
$e^{tA}$ has the form
\[ e^{tA} = e^{\lambda_1 t}A_1+e^{\lambda_2 t}A_2 ,\]
where $bc=(\lambda_1-a)(a-\lambda_2)$ and
\[ A_1 = \frac{1}{\lambda_1-\lambda_2} \begin{bmatrix} a-\lambda_2 & b \\
 c & \lambda_1-a \end{bmatrix}, \quad A_2 = \frac{1}{\lambda_1-\lambda_2} \begin{bmatrix} \lambda_1-a & -b \\
 -c & a-\lambda_2 \end{bmatrix}. \]
Without loss of generality, assume $\lambda_1 > 0 > \lambda_2$; the cases for $\lambda_1 > \lambda_2 > 0$ and $0 > \lambda_1 > \lambda_2$ are similar and thus omitted. 
Given an arbitrary $\varepsilon>0$, suppose a vector function $\bs{\phi}$ satisfies 
\[ \left\|\bs{\phi}'(t)-A\bs{\phi}(t)\right\|_{\infty} \le \varepsilon, \quad t\in\R. \]
Then, there exists a vector function $\bs{q}$ satisfying $\|\bs{q}(s)\|_{\infty} \le \varepsilon$ for all $s\in\R$, such that
$\bs{\phi}'(t)-A\bs{\phi}(t)=\bs{q}(t)$; by the variation of parameters formula,\begin{align*}
 \bs{\phi}(t)&=e^{tA}\bs{\phi}_0 + \int_0^t e^{(t-s)A}\bs{q}(s)ds \\
  &= A_1\left(e^{\lambda_1 t}\bs{\phi}_0 + \int_0^t e^{\lambda_1 (t-s)}\bs{q}(s)ds\right) + A_2\left(e^{\lambda_2 t}\bs{\phi}_0 + \int_0^t e^{\lambda_2 (t-s)}\bs{q}(s)ds\right),
\end{align*}
where $\bs{\phi}(0)=\bs{\phi}_0$. Note that
\[ \bs{x}_1 = \bs{\phi}_0 + \int_0^{\infty} e^{-\lambda_1 s}\bs{q}(s)ds \quad \text{and}\quad \bs{x}_2 = \bs{\phi}_0 - \int_{-\infty}^0 e^{-\lambda_2 s}\bs{q}(s)ds \]
are well defined (finite) in the case $\lambda_1 > 0 > \lambda_2$. Now, we consider the function
$$ \bs{x}(t) = e^{\lambda_1 t}A_1\bs{x}_1 + e^{\lambda_2 t}A_2\bs{x}_2. $$
Then, $\bs{x}$ is a solution of \eqref{maineq}. In fact, since
$$ A_2A_1= \begin{bmatrix} 0 & 0 \\ 0 & 0 \end{bmatrix} =A_1A_2 $$
holds, we have
\begin{align*}
 \bs{x}'(t)-A\bs{x}(t) &= (\lambda_1I_2-A)e^{\lambda_1 t}A_1\bs{x}_1 + (\lambda_2I_2-A)e^{\lambda_2 t}A_2\bs{x}_2\\
 &= \left(\lambda_1-\lambda_2\right)\left(e^{\lambda_1 t}A_2A_1\bs{x}_1 - e^{\lambda_2 t}A_1A_2\bs{x}_2\right) = \bs{0}.
\end{align*}
Hence,
\begin{align*}
 &\|\bs{\phi}(t)-\bs{x}(t)\|_{\infty} = \left\| -A_1\int_t^{\infty} e^{\lambda_1 (t-s)}\bs{q}(s)ds + A_2\int_{-\infty}^t e^{\lambda_2 (t-s)}\bs{q}(s)ds \right\|_{\infty}\\
 &\le \varepsilon \left(\|A_1\|_{\infty}\int_t^{\infty} e^{\lambda_1 (t-s)}ds +   \|A_2\|_{\infty}\int_{-\infty}^t e^{\lambda_2 (t-s)}ds \right) \\
 &= \frac{\varepsilon\left(-\lambda_2\|A_1\|_{\infty}+\lambda_1\|A_2\|_{\infty}\right)}{-\lambda_1\lambda_2} \\
 &= \frac{\varepsilon\left(|\lambda_2|\max\{|a-\lambda_2|+|b|, |c|+|\lambda_1-a| \}+\lambda_1\max\{|\lambda_1-a|+|b|, |c|+|a-\lambda_2|\right\}}{\lambda_1|\lambda_2|(\lambda_1-\lambda_2)}.
\end{align*}
The proof is complete.
\end{proof}

\begin{theorem}[Complex with Nonzero Real Part $\implies$ Stable]\label{thm-complex}
If $A$ has eigenvalues $\lambda=\alpha\pm i\beta$, then \eqref{maineq} is Hyers-Ulam stable on $\R$, with HUS constant 
\[ K=\frac{\sqrt{\beta^2 +\left(|a-\alpha|+\max\{|b|, |c|\}\right)^2}}{|\alpha|\beta}. \]
\end{theorem}

\begin{proof}
Consider \eqref{maineq} such that $A$ has form \eqref{Aformcomplex},
with $(\alpha-a)^2+bc+\beta^2=0$, where $\alpha\ne 0$ and $\beta>0$. Note that $A$ has eigenvalues $\lambda=\alpha\pm i\beta$. 
First, assume $\alpha>0$. We proceed with $e^{tA}$ in \eqref{complexev}.
Given an arbitrary $\varepsilon>0$, suppose a vector function $\bs{\phi}$ satisfies 
\[ \left\|\bs{\phi}'(t)-A\bs{\phi}(t)\right\|_{\infty} \le \varepsilon, \quad t\in\R. \]
Then, there exists a vector function $\bs{q}$ satisfying $\|\bs{q}(s)\|_{\infty} \le \varepsilon$ for all $s\in\R$, such that
$\bs{\phi}'(t)-A\bs{\phi}(t)=\bs{q}(t)$; by the variation of parameters formula,
$$ \bs{\phi}(t)=e^{tA}\bs{\phi}_0 + \int_0^t e^{(t-s)A}\bs{q}(s)ds, $$
for $\bs{\phi}(0)=\bs{\phi}_0$.
As the general solution of \eqref{maineq} is $\bs{x}(t)=e^{tA}\bs{x}_0$ for constant vector $\bs{x}_0$, pick 
$$ \bs{x}(t)=e^{tA}\left(\bs{\phi}_0 + \int_0^{\infty}e^{-sA}\bs{q}(s)ds\right), \qquad \bs{x}_0=\bs{\phi}_0 + \int_0^{\infty}e^{-sA}\bs{q}(s)ds, $$
which is well defined (finite) in this case, as $\alpha>0$. Then, we have
\begin{align*}
 &\|\bs{\phi}(t)-\bs{x}(t)\|_{\infty} \\
 &=\left\|-\int_t^{\infty}e^{(t-s)A}\bs{q}(s)ds\right\|_{\infty} \\
 &=\left\|\int_t^{\infty} e^{\alpha(t-s)} \begin{bmatrix} \left\{\cos(\beta(t-s))+\frac{a-\alpha}{\beta}\sin(\beta(t-s))\right\}q_1(s)+\frac{b}{\beta}\sin(\beta(t-s))q_2(s) \\ \frac{c}{\beta}\sin(\beta(t-s))q_1(s)+\left\{\cos(\beta(t-s))+\frac{\alpha-a}{\beta}\sin(\beta(t-s))\right\}q_2(s) \end{bmatrix} ds\right\|_{\infty} \\
 &=\frac{1}{\beta}\int_t^{\infty} \left\|e^{\alpha(t-s)} \begin{bmatrix} \beta q_1(s)\cos(\beta(t-s))+\left\{(a-\alpha)q_1(s)+bq_2(s)\right\}\sin(\beta(t-s)) \\ \beta q_2(s)\cos(\beta(t-s))+\left\{(\alpha-a)q_2(s)+cq_1(s)\right\}\sin(\beta(t-s)) \end{bmatrix} \right\|_{\infty}ds \\
 &\le \frac{1}{\beta}\max\begin{cases} \int_t^{\infty} e^{\alpha(t-s)} \sqrt{\beta^2 q_1^2(s)+\left\{(a-\alpha)q_1(s)+bq_2(s)\right\}^2}ds, \\ \int_t^{\infty} e^{\alpha(t-s)} \sqrt{\beta^2 q_2^2(s)+\left\{(\alpha-a)q_2(s)+cq_1(s)\right\}^2}ds \end{cases} \\
 &\le \frac{\varepsilon}{\beta}\max\begin{cases} \int_t^{\infty} e^{\alpha(t-s)} \sqrt{\beta^2 +\left(|a-\alpha|+|b|\right)^2}ds, \\ \int_t^{\infty} e^{\alpha(t-s)} \sqrt{\beta^2 +\left(|a-\alpha|+|c|\right)^2}ds \end{cases} \\
 &= \frac{\varepsilon}{\alpha\beta}\sqrt{\beta^2 +\left(|a-\alpha|+\max\{|b|, |c|\}\right)^2}.
\end{align*}
Consequently, \eqref{maineq} is Hyers-Ulam stable, with HUS constant 
$$K=\frac{\sqrt{\beta^2 +\left(|a-\alpha|+\max\{|b|, |c|\}\right)^2}}{\alpha\beta},$$
when $\alpha>0$. Suppose $\alpha<0$. As the general solution of \eqref{maineq} is $\bs{x}(t)=e^{tA}\bs{x}_0$ for constant vector $\bs{x}_0$, pick 
$$ \bs{x}(t)=e^{tA}\left(\bs{\phi}_0 - \int_{-\infty}^0 e^{-sA}\bs{q}(s)ds\right), \qquad \bs{x}_0=\bs{\phi}_0 - \int_{-\infty}^0 e^{-sA}\bs{q}(s)ds, $$
which is well defined (finite) in this case, as $\alpha<0$. Then, we have
\begin{align*}
 &\|\bs{\phi}(t)-\bs{x}(t)\|_{\infty} \\
 &=\left\|\int_{-\infty}^t e^{(t-s)A}\bs{q}(s)ds\right\|_{\infty} \\
 &=\left\|\int_{-\infty}^t e^{\alpha(t-s)} \begin{bmatrix} \left\{\cos(\beta(t-s))+\frac{a-\alpha}{\beta}\sin(\beta(t-s))\right\}q_1(s)+\frac{b}{\beta}\sin(\beta(t-s))q_2(s) \\ \frac{c}{\beta}\sin(\beta(t-s))q_1(s)+\left\{\cos(\beta(t-s))+\frac{\alpha-a}{\beta}\sin(\beta(t-s))\right\}q_2(s) \end{bmatrix} ds\right\|_{\infty} \\
 &=\frac{1}{\beta}\int_{-\infty}^t \left\|e^{\alpha(t-s)} \begin{bmatrix} \beta q_1(s)\cos(\beta(t-s))+\left\{(a-\alpha)q_1(s)+bq_2(s)\right\}\sin(\beta(t-s)) \\ \beta q_2(s)\cos(\beta(t-s))+\left\{(\alpha-a)q_2(s)+cq_1(s)\right\}\sin(\beta(t-s)) \end{bmatrix} \right\|_{\infty}ds \\
 &\le \frac{1}{\beta}\max\begin{cases} \int_{-\infty}^t e^{\alpha(t-s)} \sqrt{\beta^2 q_1^2(s)+\left\{(a-\alpha)q_1(s)+bq_2(s)\right\}^2}ds, \\ \int_{-\infty}^t e^{\alpha(t-s)} \sqrt{\beta^2 q_2^2(s)+\left\{(\alpha-a)q_2(s)+cq_1(s)\right\}^2}ds \end{cases} \\
 &\le \frac{\varepsilon}{\beta}\max\begin{cases} \int_{-\infty}^t e^{\alpha(t-s)} \sqrt{\beta^2 +\left(|a-\alpha|+|b|\right)^2}ds, \\ \int_{-\infty}^t e^{\alpha(t-s)} \sqrt{\beta^2 +\left(|a-\alpha|+|c|\right)^2}ds \end{cases} \\
 &= \frac{\varepsilon}{-\alpha\beta}\sqrt{\beta^2 +\left(|a-\alpha|+\max\{|b|, |c|\}\right)^2}.
\end{align*}
Consequently, \eqref{maineq} is Hyers-Ulam stable, with HUS constant 
$$K=\frac{\sqrt{\beta^2 +\left(|a-\alpha|+\max\{|b|, |c|\}\right)^2}}{-\alpha\beta},$$
when $\alpha<0$. The overall result follows.
\end{proof}

\section{Instability and Main Result}

\begin{lemma}[Zero Eigenvalue $\implies$ Unstable]\label{unstable1}
If at least one of the eigenvalues of $A$ is zero, then \eqref{maineq} is not Hyers-Ulam stable on $\R$.
\end{lemma}

\begin{proof}
Consider \eqref{maineq} with
$$ A=\begin{bmatrix} a & b \\ c & \lambda-a \end{bmatrix}, $$
where $bc=a(\lambda-a)$ for some real number $\lambda$. Then, $\lambda_1=\lambda$ and $\lambda_2=0$ are eigenvalues of $A$, and
\[ e^{tA} = \begin{cases} \dfrac{1}{\lambda}\begin{bmatrix} ae^{\lambda t}+\lambda-a & b(e^{\lambda t}-1) \\ c(e^{\lambda t}-1) & (\lambda-a)e^{\lambda t}+a \end{bmatrix} & \text{if}\;\lambda\ne 0, \\ \begin{bmatrix} 1+at & bt \\ ct & 1-at \end{bmatrix} & \text{if}\;\lambda=0.
\end{cases} \]
First, suppose $\lambda\ne 0$. If $a=b=0$, then 
\[ A = \begin{bmatrix} 0 & 0 \\ c & \lambda \end{bmatrix}, \quad 
   e^{tA} = \begin{bmatrix} 1 & 0 \\ \frac{c}{\lambda}(e^{\lambda t}-1) & e^{\lambda t} \end{bmatrix}. \]
Let $m:=\max\left\{1,\left|-\frac{c}{\lambda}\right|\right\}$, $\varepsilon>0$ be arbitrary, and set $\bs{\phi}(t)=\frac{\varepsilon t}{m}e^{tA} \begin{bmatrix}
1 \\ -\frac{c}{\lambda}\end{bmatrix}$. Then,
$$ \|\bs{\phi}'(t)-A\bs{\phi}(t)\|_{\infty} = \left\|\frac{\varepsilon}{m}e^{tA} \begin{bmatrix} 1 \\ -\frac{c}{\lambda}\end{bmatrix} \right\|_{\infty}
= \left\|\frac{\varepsilon}{m}\begin{bmatrix} 1 \\ -\frac{c}{\lambda}\end{bmatrix} \right\|_{\infty} = \varepsilon $$
by the choice of $m$, but for any solution $\bs{x}(t)=e^{tA}\bs{x}_0$ of \eqref{maineq},
$$ \|\bs{\phi}(t)-\bs{x}(t)\|_{\infty} = \left\|e^{tA}\left(\frac{\varepsilon t}{m} \begin{bmatrix} 1 \\ -\frac{c}{\lambda}\end{bmatrix}-\bs{x}_0 \right)\right\|_{\infty} \rightarrow\infty, \quad t\rightarrow \pm\infty, $$
for any vector $\bs{x}_0$, so that \eqref{maineq} is not Hyers-Ulam stable. Otherwise, $a\ne 0$ or $b\ne 0$ or both. Without loss of generality, assume $a\ne 0$; the case for $b\ne 0$ is similar and thus omitted. Let $m:=\max\left\{1,\left|-\frac{b}{a}\right|\right\}$, $\varepsilon>0$ be arbitrary, and set $\bs{\phi}(t)=\frac{\varepsilon t}{m}e^{tA} \begin{bmatrix} -\frac{b}{a} \\ 1 \end{bmatrix}$. Then,
$$ \|\bs{\phi}'(t)-A\bs{\phi}(t)\|_{\infty} = \left\|\frac{\varepsilon}{m}e^{tA} \begin{bmatrix} -\frac{b}{a} \\ 1 \end{bmatrix} \right\|_{\infty}
= \left\|\frac{\varepsilon}{m}\begin{bmatrix} -\frac{b}{a} \\ 1 \end{bmatrix} \right\|_{\infty} = \varepsilon $$
by the choice of $m$, but for any solution $\bs{x}(t)=e^{tA}\bs{x}_0$ of \eqref{maineq},
$$ \|\bs{\phi}(t)-\bs{x}(t)\|_{\infty} = \left\|e^{tA}\left(\frac{\varepsilon t}{m} \begin{bmatrix} -\frac{b}{a} \\ 1 \end{bmatrix}-\bs{x}_0 \right)\right\|_{\infty} \rightarrow\infty, \quad t\rightarrow \pm\infty, $$
for any vector $\bs{x}_0$, using $bc=a(\lambda-a)$, so that \eqref{maineq} is not Hyers-Ulam stable. As a result, in either case, \eqref{maineq} is not Hyers-Ulam stable when $\lambda\ne 0$.

Second, assume $\lambda=0$. Then, $$ A=\begin{bmatrix} a & b \\ c & -a \end{bmatrix}, \quad a^2+bc=0, $$
and
\[ e^{tA} = \begin{bmatrix} 1+at & bt \\ ct & 1-at \end{bmatrix}.  \]
If $b\ne -a$ or $c\ne a$ or both, let $\varepsilon>0$ be arbitrary, and set 
$\bs{\phi}(t) = \dfrac{\varepsilon t}{2} \begin{bmatrix} (a+b)t+2 \\ (c-a)t+2 \end{bmatrix}$. Then, using $a^2+bc=0$, we have
$$ \|\bs{\phi}'(t)-A\bs{\phi}(t)\|_{\infty} = \left\|\varepsilon\begin{bmatrix} 1 \\ 1 \end{bmatrix} \right\|_{\infty} = \varepsilon, $$
but for any solution $\bs{x}(t)=e^{tA}\bs{x}_0$ of \eqref{maineq},
$$ \|\bs{\phi}(t)-\bs{x}(t)\|_{\infty} = \left\|\dfrac{\varepsilon t}{2} \begin{bmatrix} (a+b)t+2 \\ (c-a)t+2 \end{bmatrix} - \begin{bmatrix} 1+at  & bt \\ ct & 1-at\end{bmatrix}\bs{x}_0 \right\|_{\infty} \rightarrow\infty, \quad t\rightarrow \pm\infty, $$
for any vector $\bs{x}_0$ since either $b\ne -a$ or $c\ne a$, so that \eqref{maineq} is not Hyers-Ulam stable. 
If $b=-a$ and $c=a$, let $\varepsilon>0$ be arbitrary, and set 
$\bs{\phi}(t) = \varepsilon t \begin{bmatrix} at+1 \\ at-1 \end{bmatrix}$. Then, 
$$ \|\bs{\phi}'(t)-A\bs{\phi}(t)\|_{\infty} = \left\|\varepsilon\begin{bmatrix} 1 \\ -1 \end{bmatrix} \right\|_{\infty} = \varepsilon, $$
but for any solution $\bs{x}(t)=e^{tA}\bs{x}_0$ of \eqref{maineq},
$$ \|\bs{\phi}(t)-\bs{x}(t)\|_{\infty} = \left\|\varepsilon t \begin{bmatrix} at+1 \\ at-1 \end{bmatrix} - \begin{bmatrix} 1+at  & -at \\ at & 1-at\end{bmatrix}\bs{x}_0 \right\|_{\infty} \rightarrow\infty, \quad t\rightarrow \pm\infty, $$
for any vector $\bs{x}_0$ and any $a\in\R$, so that \eqref{maineq} is not Hyers-Ulam stable in this case either. Thus, \eqref{maineq} is not Hyers-Ulam stable if $\lambda=0$. Putting all the cases and subcases together, the result follows and the proof is complete. 
\end{proof}

\begin{lemma}[Purely Imaginary $\implies$ Unstable]\label{unstable2}
If the eigenvalues of $A$ are purely imaginary, then \eqref{maineq} is not Hyers-Ulam stable on $\R$.
\end{lemma}

\begin{proof}
Consider \eqref{maineq} with
$$ A=\begin{bmatrix} a & b \\ c & -a \end{bmatrix}, $$
where $a^2+bc+\beta^2=0$ for some real number $\beta>0$.
Note that $A$ has eigenvalues $\pm \beta i$. Using Putzer's algorithm \cite[Theorem 2.35]{kp} or \eqref{complexev}, we have
\[ e^{tA} = \begin{bmatrix} \cos(\beta t)+\frac{a}{\beta}\sin(\beta t) & \frac{b}{\beta}\sin(\beta t) \\ \frac{c}{\beta}\sin(\beta t) & \cos(\beta t)-\frac{a}{\beta}\sin(\beta t) 
\end{bmatrix}. \]
Let $m:=\max\left\{\frac{1}{\beta}\sqrt{a^2+\beta^2},\, \frac{|c|}{\beta}\right\}$. Given an arbitrary $\varepsilon>0$, let $\bs{\phi}(t):=\frac{\varepsilon t}{m}e^{tA}\begin{bmatrix} 1 \\ 0 \end{bmatrix}$. Then, $\bs{\phi}$ satisfies
\[ \bs{\phi}'(t)-A\bs{\phi}(t)=\frac{\varepsilon}{m}e^{tA}\begin{bmatrix} 1 \\ 0 \end{bmatrix}, \]
so that
\[ \left\|\bs{\phi}'(t)-A\bs{\phi}(t)\right\|_{\infty}=\left\|\frac{\varepsilon}{m}e^{tA}\begin{bmatrix} 1 \\ 0 \end{bmatrix}\right\|_{\infty}=\varepsilon. \]
As the general solution of \eqref{maineq} is $\bs{x}(t)=e^{tA}\bs{x}_0$ for constant vector $\bs{x}_0$, we have
\[ \|\bs{\phi}(t)-\bs{x}(t)\|_{\infty} = \left\|e^{tA}\left(\frac{\varepsilon t}{m}\begin{bmatrix} 1 \\ 0 \end{bmatrix}-\bs{x}_0\right) \right\|_{\infty} \longrightarrow\infty \quad\text{as}\quad {t\rightarrow\pm\infty} \]
for any constant vector $\bs{x}_0$, making \eqref{maineq} Hyers-Ulam unstable in this case.
\end{proof}

\begin{theorem}[Main Result]\label{thm-main}
Equation \eqref{maineq} is Hyers-Ulam stable on $\R$ if and only if $A$ has eigenvalues with non-zero real part. The best (minimal) Hyers-Ulam constant $K$ satisfies $K\ge \left\|A^{-1}\right\|_{\infty}$, with equality in some cases. In particular, if $A$ has distinct eigenvalues $\lambda_1 > \lambda_2$ with $\lambda_1 \lambda_2 > 0$ and $\lambda_1 \ge a \ge \lambda_2$, or $A$ has a repeated eigenvalue $\lambda\ne 0$ with $\lambda=a$ ($bc=0$), then the best (minimal) Hyers-Ulam constant is $\left\|A^{-1}\right\|_{\infty}$.
\end{theorem}

\begin{proof}
The stability part of the theorem follows from Theorems \ref{thm-cases}, \ref{thm-distinct-Nonzero}, \ref{thm-complex}. 
The instability part of the theorem follows from Lemmas \ref{unstable1} and \ref{unstable2}. 
The minimal Hyers-Ulam constant $K$ satisfies $K\ge \left\|A^{-1}\right\|_{\infty}$ by Lemma \ref{lower-bound}, with equality holding in the cases shown in Theorem \ref{thm-cases}.

Next, we will show that if $A$ has distinct eigenvalues $\lambda_1 > \lambda_2$ with $\lambda_1 \lambda_2 > 0$ and $\lambda_1 \ge a \ge \lambda_2$, or $A$ has a repeated eigenvalue $\lambda\ne 0$ with $\lambda=a$ ($bc=0$), then the best (minimal) Hyers-Ulam constant is $\left\|A^{-1}\right\|_{\infty}$. Let
$$ \bs{e} = \begin{bmatrix} u_1 \\ u_2 \end{bmatrix} $$
be a unit vector. Note here that
$$ \left\|A^{-1}\bs{e}\right\|_{\infty} = \frac{1}{|\lambda_1\lambda_2|}\max\{|(\lambda_1+\lambda_2-a)u_1-bu_2|, |-cu_1+au_2| \} $$
and
$$ \left\|A^{-1}\right\|_{\infty} = \frac{1}{|\lambda_1\lambda_2|}\max\{|\lambda_1+\lambda_2-a|+|b|, |c|+|a|\} $$
if $A$ has non-zero eigenvalues $\lambda_1 > \lambda_2$ with $\lambda_1 \lambda_2 > 0$ and $\lambda_1 \ge a \ge \lambda_2$; and
$$ \left\|A^{-1}\bs{e}\right\|_{\infty} = \frac{1}{\lambda^2}\max\{|\lambda u_1-bu_2|, |-cu_1+\lambda u_2| \} $$
and
$$ \left\|A^{-1}\right\|_{\infty} = \frac{|\lambda|+\max\{|b|, |c|\}}{\lambda^2} $$
if $A$ has a repeated eigenvalue $\lambda\ne 0$ with $\lambda=a$ ($bc=0$). 

Now we consider the first case. Suppose $\lambda_1 \ge a \ge \lambda_2 > 0$ and $b$, $c\ge0$. Put $u_1=1$ and $u_2=-1$. Then, we have
$$ \left\|A^{-1}\bs{e}\right\|_{\infty} = \frac{1}{|\lambda_1\lambda_2|}\max\{|\lambda_1+\lambda_2-a+b|, |-c-a| \} = \left\|A^{-1}\right\|_{\infty}. $$
The other options within this first case are similar and thus omitted.

Next we consider the second case. Suppose $\lambda > 0$ and $b=0$. Put $u_1=-\mathrm{sign}(c)$ and $u_2=1$. Then we have
$$ \left\|A^{-1}\bs{e}\right\|_{\infty} = \frac{\max\{|\lambda|, |c|+\lambda\}}{\lambda^2} = \left\|A^{-1}\right\|_{\infty}. $$
Likewise, the other options within this second case follow and are omitted.
Hence, in all of the above cases, $\left\|A^{-1}\right\|_{\infty}$ is the best Hyers-Ulam constant.
\end{proof}

\section{Application to Second-Order Linear Differential Equations}


\begin{remark}\label{rem1}
Consider the second-order linear constant coefficient homogeneous differential equation
\begin{equation}\label{2nd-order-eq}
   x''(t)-(\lambda_1+\lambda_2)x'(t)+\lambda_1\lambda_2 x(t)=0, \quad t\in\R, 
\end{equation}
where $\lambda_1,\lambda_2\in\R$, or they are complex conjugates of each other if not real. If we introduce the function $u=x'$ and the vector function $\bs{x}(t)=\begin{bmatrix} x \\ u \end{bmatrix}$, then this equation is equivalent to system \eqref{maineq} with $A=\begin{bmatrix} 0 & 1 \\ -\lambda_1\lambda_2 & \lambda_1+\lambda_2\end{bmatrix}$. We can easily show that if this system is Hyers-Ulam stable, then \eqref{2nd-order-eq} is also Hyers-Ulam stable. In fact, for any $\varepsilon>0$, we assume the condition
$$ q(t)=\phi''(t)-(\lambda_1+\lambda_2)\phi'(t)+\lambda_1\lambda_2 \phi(t), \quad |q(t)| \le \varepsilon, \quad t\in\R. $$
Using the function $\psi=\phi'$ and the vector function $\bs{\phi}(t)=\begin{bmatrix} \phi \\ \psi \end{bmatrix}$, we then have
$$ \left\|\bs{\phi}'(t)- \begin{bmatrix} 0 & 1 \\ -\lambda_1\lambda_2 & \lambda_1+\lambda_2\end{bmatrix}\bs{\phi}(t)\right\|_{\infty} = \left\|\begin{bmatrix} 0 \\ q(t) \end{bmatrix}\right\|_{\infty} = |q(t)| \le \varepsilon, $$
so that, if system \eqref{maineq} is Hyers-Ulam stable, then there exists a solution $\bs{x}=\begin{bmatrix} x \\ u \end{bmatrix}$ of system \eqref{maineq} such that
$$ \max\{|\phi(t)-x(t)|, |\psi(t)-u(t)|\} = \|\bs{\phi}(t)-\bs{x}(t)\|_{\infty} \le K \varepsilon, \quad t\in\R. $$
Hence, this implies $|\phi(t)-x(t)| \le K \varepsilon$ for $t\in\R$. Since $x$ is a solution of \eqref{2nd-order-eq}, equation \eqref{2nd-order-eq} is Hyers-Ulam stable on $\R$. As we can see from the proof of Lemmas $\ref{unstable1}$ and $\ref{unstable2}$, even if we focus only on the first component of the vector function $\bs{\phi}(t)-\bs{x}(t)$, we can see that the absolute value $|\phi(t)-x(t)|$ of it diverges as $t\rightarrow \pm\infty$. That is, the instability of system \eqref{maineq} and that of equation \eqref{2nd-order-eq} are equivalent.

Suppose $\operatorname{Re}(\lambda_j)=0$ for $j=1$ or $j=2$ or both. Then, \eqref{2nd-order-eq} is not Hyers-Ulam stable, by Theorem $\ref{thm-main}$.

Suppose $\lambda_1=\alpha+i\beta$ and $\lambda_2=\alpha-i\beta$ with $\alpha\ne 0$ and $\beta>0$. Then, 
$$ A=\begin{bmatrix} 0 & 1 \\ -(\alpha^2+\beta^2) & 2\alpha \end{bmatrix}, $$
and \eqref{2nd-order-eq} is Hyers-Ulam stable on $\R$, with HUS constant
\[ K=\frac{\sqrt{\beta^2 +\left(|\alpha|+\max\{1, \alpha^2+\beta^2\}\right)^2}}{|\alpha|\beta} \]
by Theorem $\ref{thm-complex}$.

Suppose $\lambda_1=\lambda_2=\lambda>0$. Then, 
$$ A=\begin{bmatrix} 0 & 1 \\ -\lambda^2 & 2\lambda \end{bmatrix}, $$
and \eqref{2nd-order-eq} is Hyers-Ulam stable on $\R$, with HUS constant 
\[ K=\begin{cases} \frac{2\lambda+1}{\lambda^2} &\text{if}\;\; 0 < \lambda < \frac{e-1}{e} + \frac{\sqrt{1-2e+2e^2}}{e}, \\
 \frac{\lambda+2e^{-1}}{\lambda} &\text{if}\;\; \lambda \ge \frac{e-1}{e} + \frac{\sqrt{1-2e+2e^2}}{e} \end{cases} \]
by Theorem $\ref{thm-cases}$.

Suppose $\lambda_1>0>\lambda_2$. Then, \eqref{2nd-order-eq} is Hyers-Ulam stable, with Hyers-Ulam stability constant
$$ K = \frac{|\lambda_2|(|\lambda_2|+1)\max\{1,|\lambda_1|\}+|\lambda_1|(|\lambda_1|+1)\max\{|\lambda_2|,1\}}{|\lambda_1\lambda_2|(\lambda_1-\lambda_2)} $$
by Theorem $\ref{thm-distinct-Nonzero}$.

Suppose $\lambda_1>\lambda_2>0$. Then, \eqref{2nd-order-eq} is Hyers-Ulam stable, with Hyers-Ulam stability constant
$$ K_{(ii)} = \max\left\{\frac{\lambda_1+\lambda_2+1}{\lambda_1\lambda_2}, \; \frac{\lambda_1\lambda_2+2\lambda_2\left(\frac{\lambda_1}{\lambda_2}\right)^{\frac{-\lambda_2}{\lambda_1-\lambda_2}}}{\lambda_1\lambda_2}\right\} $$
by Theorem $\ref{thm-cases}$.
\end{remark}


\begin{remark}
Again, consider the second-order linear constant coefficient homogeneous differential equation
\begin{equation}\label{2nd-order-eq-2}
   x''(t)-(\lambda_1+\lambda_2)x'(t)+\lambda_1\lambda_2 x(t)=0, \quad t\in\R, 
\end{equation}
where $\lambda_1,\lambda_2\in\R$. This time we introduce the function $u=x'-\lambda_1 x$ and the vector function $\bs{x}(t)=\begin{bmatrix} x \\ u \end{bmatrix}$, so that this equation is equivalent to system \eqref{maineq} with $A=\begin{bmatrix} \lambda_1 & 1 \\ 0 & \lambda_2\end{bmatrix}$. If $\lambda_j=0$ for $j=1$ or $j=2$ or both, then \eqref{2nd-order-eq-2} is not Hyers-Ulam stable by Theorem $\ref{thm-main}$ as before. If $\lambda_1=\lambda_2=\lambda>0$, then \eqref{2nd-order-eq-2} is Hyers-Ulam stable on $\R$, with HUS constant 
\[ K=\frac{1+\lambda}{\lambda^2}=\left\|A^{-1}\right\|_{\infty} \]
by Theorem $\ref{thm-cases}$. Moreover, since $\lambda = a$ is satisfied, this constant is the best constant for system \eqref{maineq} with $A=\begin{bmatrix} \lambda & 1 \\ 0 & \lambda\end{bmatrix}$ by Theorem $\ref{thm-main}$.

If $\lambda_1>\lambda_2>0$ or $0>\lambda_1>\lambda_2$, then \eqref{2nd-order-eq-2} is Hyers-Ulam stable, with Hyers-Ulam stability constant
$$ K = \left\|A^{-1}\right\|_{\infty} $$
by Theorem $\ref{thm-cases}$. Moreover, since $\lambda_1 = a$ is satisfied, this constant is the best constant for system \eqref{maineq} with $A=\begin{bmatrix} \lambda_1 & 1 \\ 0 & \lambda_2\end{bmatrix}$ by Theorem $\ref{thm-main}$.
 
If $\lambda_1$ and $\lambda_2$ are nonzero real numbers with $\lambda_1 > 0 > \lambda_2$,
then \eqref{2nd-order-eq-2} is Hyers-Ulam stable on $\R$, with HUS constant
$$ K = \frac{|\lambda_2|(\lambda_1-\lambda_2+1)+|\lambda_1|\max\left\{1, \lambda_1-\lambda_2\right\}}{|\lambda_1\lambda_2|(\lambda_1-\lambda_2)}, $$
using Theorem $\ref{thm-distinct-Nonzero}$.

Note that the different substitution used in this remark as compared with the previous remark leads to different HUS constants $K$ for the same equation. This is thought to be due to the effect of the second component $\psi(t)-u(t)$ of the vector function $\bs{\phi}(t)-\bs{x}(t)$ introduced in Remark $\ref{rem1}$.
\end{remark}

\begin{example}
Let $A=\begin{bmatrix} 2 & 1 \\ 1 & 2 \end{bmatrix}$. Note that $A^{-1}=\begin{bmatrix} \frac{2}{3} & -\frac{1}{3} \\ -\frac{1}{3} & \frac{2}{3} \end{bmatrix}$, $\lambda_1=3$, $a=2$, and $\lambda_2=1$, so $\lambda_1>a>\lambda_2$. By Theorem $\ref{thm-cases}$, \eqref{maineq} is Hyers-Ulam stable with HUS constant $K=\left\|A^{-1}\right\|_{\infty}=1$.

Let $A=\begin{bmatrix} 0 & 1 \\ -2 & 3 \end{bmatrix}$. Note that $A^{-1}=\begin{bmatrix} \frac{3}{2} & -\frac{1}{2} \\ 1 & 0 \end{bmatrix}$, $\lambda_1=2$, $a=0$, and $\lambda_2=1$, so $\lambda_1>\lambda_2>a$. By Theorem $\ref{thm-cases}$, \eqref{maineq} is Hyers-Ulam stable with HUS constant $K=\max\{2,\frac{3}{2}\}=2=\left\|A^{-1}\right\|_{\infty}$.

Let $A=\begin{bmatrix} 3 & 1 \\ -2 & 0 \end{bmatrix}$. Note that $A^{-1}=\begin{bmatrix} 0 & -\frac{1}{2} \\ 1 & \frac{3}{2} \end{bmatrix}$, $\lambda_1=2$, $a=3$, and $\lambda_2=1$, so $a>\lambda_1>\lambda_2$. By Theorem $\ref{thm-cases}$, \eqref{maineq} is Hyers-Ulam stable with HUS constant $K=\max\{1,\frac{5}{2}\}=\frac{5}{2}=\left\|A^{-1}\right\|_{\infty}$.

Let $A=\begin{bmatrix} 2 & -1 \\ 3 & -2 \end{bmatrix}=A^{-1}$. Note that $\lambda_1=1$ and $\lambda_2=-1$, so $\lambda_1>0>\lambda_2$. By Theorem $\ref{thm-distinct-Nonzero}$, \eqref{maineq} is Hyers-Ulam stable with HUS constant $K=\frac{\max\{4,4\}+\max\{2,6\}}{2}=5=\left\|A^{-1}\right\|_{\infty}$.

For each of the instances in this example, the HUS constant $K$ is the best possible by Theorem $\ref{thm-main}$. Additionally, these best HUS constants significantly improve the results given in \cite[Theorem 3]{bmppr} and \cite[Corollary 2]{jung2}, which have variable expressions rather than an HUS constant.
\end{example}

\section{Significance of the Results}
In conclusion, new necessary and sufficient conditions for a homogeneous linear differential system with 2 $\times$ 2 constant coefficient matrix to be Hyers-Ulam stable are proven. Moreover, for the first time, the best (minimal) Hyers-Ulam constant for such systems is found in some cases, along with a lower bound for the best constant in all stability cases. Obtaining the best Hyers-Ulam constant for second-order constant coefficient differential equations illustrates the applicability of the strong results. One future direction would be to find necessary and sufficient conditions for HUS of general linear systems, and the best HUS constant(s).

\section*{Funding}
The second author was supported by JSPS KAKENHI Grant Number JP20K03668.

\section*{Competing Interests}
The authors declare that they have no conflict of interest or competing interests.

\section*{Author Contributions}
Both authors contributed equally to the results in this paper.



\begin{thebibliography}{999}

\bibitem{rpp} Aruldass, A.R., Pachaiyappan, D., Park, C.: Hyers-Ulam stability of second-order differential equations using Mahgoub transform. Adv. Difference Equ. 2021, Paper No. 23, 10 pp (2021).

\bibitem{bp} Baias, A., Popa, D.: On the best Ulam constant of the second order linear differential operator. Rev. R. Acad. Cienc. Exactas F\'{i}s. Nat. Ser. A Mat. RACSAM 114, no. 1, Paper No. 23, 15 pp (2020).

\bibitem{bclo} Bu\c{s}e, C., Lupulescu, V., O'Regan, D.: Hyers-Ulam stability for equations with differences and differential equations with time-dependent and periodic coefficients. Proc. Roy. Soc. Edinburgh Sect. A 150, no. 5, 2175--2188 (2020).

\bibitem{bmppr} Blaga, F., Mesaro\c{s}, L., Popa, D., Pugna, G., Ra\c{s}a, I.: Bounds for solutions of linear differential equations and Ulam stability. Miskolc Math. Notes 21, no. 2, 653--664 (2020).

\bibitem{dragicevic} Dragi\v{c}evi\'{c}, D.: Hyers-Ulam stability for a class of perturbed Hill's equations. Results Math. 76, no. 3, Paper No. 129, 11 pp (2021).

\bibitem{fo} Fukutaka, R., Onitsuka, M.: Best constant for Ulam stability of Hill's equations. Bull. Sci. Math. 163, 102888, 23pp (2020).

\bibitem{jung} Jung, S.-M.: Hyers-Ulam stability of a system of first order linear differential equations with constant coefficients. J. Math. Anal. Appl. 320 549--561 (2006).

\bibitem{jung2} Jung, S-M.: Hyers-Ulam stability of the first-order matrix differential equations. J. Function Spaces (2015), Article ID 614745, 7 pages.
http://dx.doi.org/10.1155/2015/614745

\bibitem{jung3} Jung, S.-M., Nam, Y.W.: Hyers-Ulam stability of the first order inhomogeneous matrix difference equation. J. Comput. Anal. Appl. 23 No. 8, 1368--1383 (2017).

\bibitem{kp} Kelley, W., Peterson, A.: The Theory of Differential Equations: Classical and Qualitative. Pearson Prentice Hall, Upper Saddle River, NJ (2004).

\bibitem{mp1} Murali. R., Selvan, A.P.: Hyers-Ulam stability of $n$th order linear differential equation. Proyecciones 38, no. 3, 553--566 (2019).

\bibitem{mp2} Murali, R., Selvan, A.P.: Hyers-Ulam-Rassias stability for the linear ordinary differential equation of third order. Kragujevac J. Math. 42, no. 4, 579--590 (2018).

\bibitem{mps} Murali, R., Park, C., Selvan, A.P.: Hyers-Ulam stability for an $n$th order differential equation using fixed point approach. J. Appl. Anal. Comput. 11, no. 2, 614--631 (2021).

\bibitem{msp} Murali, R., Selvan, A.P., Park, C.: Ulam stability of linear differential equations using Fourier transform. AIMS Math. 5, no. 2, 766--780 (2020).

\bibitem{sl} Shen, Y., Li, Y.: A general method for the Ulam stability of linear differential equations. Bull. Malays. Math. Sci. Soc. 42, no. 6, 3187--3211 (2019).

\bibitem{ugbrgvla} Unyong, B., Govindan, V., Bowmiya, S., Rajchakit, G., Gunasekaran, N., Vadivel, R., Lim, C., Agarwal, P.: Generalized linear differential equation using Hyers-Ulam stability approach. AIMS Math. 6, no. 2, 1607--1623 (2021).

\bibitem{ym} Yang, Y., Meng, F.: A kind of stricter Hyers-Ulam stability of second order linear differential equations of Carath\'{e}odory type. Appl. Math. Lett. 115, 106946, 7 pp (2021).

\end{thebibliography}
\end{document}